\documentclass{amsart}
\usepackage{amssymb}
\usepackage{eucal}
\usepackage{amsfonts}
\usepackage{epsfig}
\usepackage[all]{xy}

\usepackage{hyperref}
\let\oldcite\cite                                  

\newtheorem{thm}{Theorem}[section]
\newtheorem{cor}[thm]{Corollary}
\newtheorem{lem}[thm]{Lemma}
\newtheorem{prop}[thm]{Proposition}
\theoremstyle{definition}
\newtheorem{defn}[thm]{Definition}
\theoremstyle{remark}
\newtheorem{rem}[thm]{Remark}
\numberwithin{equation}{section} \theoremstyle{remark}
\newtheorem{ex}[thm]{Example}

\newcommand{\bbG}{\mathbb{G}}
\newcommand{\bbH}{\mathbb{H}}

\newcommand{\bbK}{\mathbb{K}}
\newcommand{\bbU}{\mathbb{U}}
\newcommand{\bbV}{\mathbb{V}}
\newcommand{\bbW}{\mathbb{W}}
\newcommand{\bbE}{\mathbb{E}}

\let\:=\colon

\newcommand{\lra}{\longrightarrow}
\newcommand{\llra}[1]{\stackrel{#1}{\lra}}

\newcommand{\Int}{\operatorname{Int}}

\newcommand{\id}{\operatorname{id}}

\newcommand{\coker}{\operatorname{coker}}

\newcommand{\pr}{\operatorname{pr}}
\newcommand{\Lie}{\operatorname{Lie}}
\newcommand{\Hfib}{\operatorname{hfib}}

\newcommand{\TwoTerm}{\mathbf{2TermL_{\infty}}}
\newcommand{\TwoTermB}{\mathbf{2TermL_{\infty}^{\flat}}}
\newcommand{\LieXM}{\mathbf{LieXM}}

\newcommand{\LieAlgXM}{\mathbf{LieAlgXM}}
\newcommand{\Linf}{L_{\infty}}
\newcommand{\ang}[1]{\langle #1\rangle}

\def\smashedlongrightarrow{\setbox0=\hbox{$\longrightarrow$}\ht0=1pt\box0}
\def\risom{\buildrel\sim\over{\smashedlongrightarrow}}

\newcommand{\oux}[2]{\underset{#1}{\overset{#2}\times}}
\newcommand{\ouplus}[2]{\underset{#1}{\overset{#2}\oplus}}

\def\lie2{$2$-term $\Linf$-algebra}

\begin{document}

\title[Integrating morphisms of Lie $2$-algebras]
      {Integrating morphisms of Lie $2$-algebras}%

\author{Behrang Noohi}%

\begin{abstract}
   We show how to integrate a weak morphism of Lie algebra crossed-modules  to 
   a weak morphism of Lie 2-groups. To do so we develop a theory of butterflies for \lie2s. 
   In particular,  we obtain a new description of the bicategory of \lie2s. We use butterflies
   to  give a functorial construction of connected covers of Lie 2-groups. We also discuss the 
   notion of homotopy fiber of a morphism of \lie2s.
\end{abstract}
\maketitle
\section{Introduction}{\label{S:Intro}}

In this paper, we tackle two main problems in the Lie theory of 2-groups: 1) integrating 
weak morphisms of Lie 2-algebras to weak morphisms of Lie 2-groups; 2) functorial 
construction of connected covers of Lie 2-groups. As we will see, the latter plays
an important role in the solution
of the former. Let us explain (1) and (2) in detail and outline our solution to them. 

\medskip

\noindent{\em Problem (1).}
A weak morphism $f \: \bbH \to \bbG$ of  Lie 2-groups gives rise to a weak morphism 
of  Lie 2-algebras $\Lie f \: \Lie\bbH \to \Lie\bbG$. (If we regard $\Lie\bbH$
and $\Lie\bbG$ as \lie2s, $\Lie f$ is then a morphism of \lie2s in the sense of Definition 
\ref{D:morphismlie2alg}.) Problem (1) can be stated as follows: given a morphism
$F \: \Lie\bbH \to \Lie\bbG$ of Lie 2-algebras can we integrate it to a weak morphism 
$\Int F  \: \bbH \to \bbG$ of Lie 2-groups? 

We answer this question affirmatively by the following
theorem (see  Theorem \ref{T:bifunctor} for a more precise statement).

 \begin{thm}{\label{T:integration}}
   Let $\bbG$ and $\bbH$ be (strict) Lie 2-groups. Suppose that $\bbH$ is
   2-connected (Definition \ref{D:connected}). Then, to give a weak morphism 
   $f \: \bbH\to \bbG$ is equivalent to giving a morphism of Lie 2-algebras 
   $\Lie f \: \Lie\bbH \to \Lie\bbG$.  The same thing is true for 2-morphisms.
 \end{thm}
 
This theorem is the 2-group version of the well-known fact in Lie theory that a Lie
homomorphism $f \: H \to G$ is uniquely given by its effect on Lie algebras
$\Lie f \: \Lie H \to \Lie G$, whenever $H$ is 1-connected.
It implies the following (see Corollary \ref{C:adjoint}).

 \begin{thm}{\label{T:main}}
      The bifunctor $\Lie \: \LieXM \to \LieAlgXM$ has a left adjoint
       $$\Int \: \LieAlgXM \to \LieXM.$$
 \end{thm}

Here, $\LieXM$ is the bicategory of Lie crossed-modules and weak morphisms, and 
$\LieAlgXM$ is the bicategory of Lie algebra crossed-modules and 
weak morphisms.\footnote{Note that $\LieXM$ is naturally biequivalent to 
the 2-category of strict Lie 2-groups and
weak morphisms, and this is in turn biequivalent to the 2-category of strict 
Lie group stacks. The bicategory  $\LieAlgXM$ is naturally biequivalent to 
the full sub 2-category of the 2-category  $\TwoTerm$ of \lie2s consisting of
strict \lie2s, and this is in turn biequivalent to the 2-category of 2-term dglas.} 
The bifunctor $\Int$ takes a Lie crossed-module to the unique  2-connected (strict) Lie 2-group
that integrates it. When restricted to the full subcategory $\mathbf{Lie} \subset\LieAlgXM$ 
of Lie algebras, it coincides with
the standard integration functor which sends a Lie algebra $V$ to the simply-connected Lie
group $\Int V$ with Lie algebra $V$.

The problem of integrating $\Linf$-algebras has been studied in \cite{Getzler} and 
\cite{Henriques}, where they show how to integrate an $\Linf$-algebra (to a simplicial manifold).
The focus of these two papers, however, is different from ours in that we begin we {\em fixed}
Lie 2-groups $\bbH$ and $\bbG$ and study the problem of integrating a morphism of Lie
2-algebras $\Lie\bbH \to \Lie\bbG$.
The problem of integrating
morphisms of Lie 2-algebras in the case where the source is a Lie algebra
has been studied in  \cite{ZhuZambon} using the classical
approach via paths and solving PDEs. 
Our approach circumvents the necessity to use paths and solving
PDEs and is more formal, making it completely functorial (hence applicable in other
circumstances) and explicit.

\medskip\noindent{\em Problem (2).} For a Lie group $G$, its 0-th and 1-st connected covers
$G\ang{0}$ and $G\ang{1}$, which  are again Lie groups, play an important role in Lie
theory. We observe that for (strict) Lie 2-groups  one needs to go one  step further, i.e., 
one needs to consider
the 2-nd connected cover as well.  
We prove the following theorem.

 \begin{thm}{\label{T:covers}}
    For $n=0,1,2$, there are bifunctors $(-)\ang{n}\: \LieXM  \to \LieXM$  sending
    a Lie crossed-modules $\bbG$ to its $n$-th connected cover. These bifunctors
    come with natural transformations $q_n \: (-)\ang{n} \Rightarrow \id$  
    such that, for every $\bbG$, $q_n \: \bbG\ang{n} \to \bbG$  induces
    isomorphisms on $\pi_i$  for $i\geq n+1$.
    Furthermore, $(-)\ang{n}$ is right adjoint to the inclusion of the full sub bicategory 
    of $n$-connected Lie crossed-modules in $\LieXM$.
 \end{thm}

The above theorem is essentially the content of Sections \ref{S:Connected}--\ref{S:Functoriality}.
We will be especially interested in 
the 2-connected cover $\bbG\ang{2}$ because, as suggested by Theorem \ref{T:main}, it 
seems to be the correct replacement for the universal cover of a Lie group in the
Lie theory of 2-groups.

\medskip
\noindent{\em Method.} To solve (1) and (2) we employ the
machinery of {\em butterflies}, which we believe
is of independent interest. Roughly speaking, a butterfly 
 (Definition \ref{D:butterfly}) between \lie2s  is a Lie algebra theoretic
version of a Morita morphism.
We use butterflies 
to give a new description of the  
2-category $\TwoTerm$ of \lie2s introduced in \cite{BaezCrans}. 
The advantage of using butterflies
is twofold.
On the one hand, butterflies do away with cumbersome cocycle formulas and are
much easier to manipulate. On the other hand, given the diagrammatic nature of butterflies,
they are better adapted to geometric situations; this is what allows us to prove Theorem
\ref{T:main}.

Butterflies for \lie2s parallel the corresponding theory for Lie 2-groups  developed in \cite{Maps}
(see $\S$9.6 therein) and \cite{ButterflyI}. In fact, taking Lie algebras converts a butterfly
in Lie groups to a butterfly in Lie algebras ($\S$\ref{S:Bifun}). 
This allows us to study weak morphisms of Lie 2-groups
using butterflies between \lie2s, thereby reducing the problem to one about extensions of 
Lie algebras. With Theorem \ref{T:main} at hand, we expect that this 
provide a convenient framework for studying weak morphisms of Lie 2-groups.

\medskip\noindent
{\em Organization of the paper.}
Sections \ref{S:lie2}--\ref{S:Bicat} are devoted to setting up the machinery of butterflies
and constructing the bicategory $\TwoTermB$ of \lie2s and butterflies.
We show that $\TwoTermB$ is biequivalent to the Baez-Crans 2-category $\TwoTerm$
of \lie2s.  In $\S$\ref{S:Cone}  we  discuss the {\em homotopy fiber} of a morphism of \lie2s. 
The homotopy fiber is the  Lie algebra counterpart of what we called the homotopy 
fiber of a weak morphism of Lie 2-groups in \cite{Maps}, $\S$9.4. 
The homotopy fiber of $f$ measures the deviation of $f$ from being an
equivalence and it sits in a natural exact triangle which gives rise to a 7-term long
exact sequence. The homotopy fiber  comes with a rich structure consisting of 
various brackets and Jacobiators (see $\S$\ref{SS:structure}).
We are not aware whether this structure has been previously studied. It is presumably 
some kind of a Lie algebra version of what is called a `crossed-module
in groupoids' in  \cite{Brown-Gilbert}.

In Section \ref{S:Weak} we review Lie 2-groups and weak morphisms (butterflies) of Lie 2-groups.
Sections \ref{S:Connected}--\ref{S:Functoriality} are devoted to the solution of Problem (2).
For a Lie crossed-modules $\bbG$ we define its $n$-th connected covers  
$\bbG\ang{n}$, for $n\leq 2$, and show that they are functorial and have the expected adjunction
property.

In Section \ref{S:Bifun} we solve Problem (1) by proving Theorems \ref{T:main}
and \ref{T:integration}. The proofs rely on the
solution of Problem (2) given in Sections \ref{S:Connected}--\ref{S:Functoriality} and the
theory of butterflies developed in Sections \ref{S:lie2}--\ref{S:Bicat}.

\medskip 

\noindent{\bf Acknowledgement.}
  I  thank Gustavo Granja,  Christoph Wockel and Chenchang Zhu for helpful conversations 
  on the subject of this paper.
  I also thank Fernando Muro, Jim Stasheff, Tim Porter, Dmitry
  Roytenberg and Marco Zambon for making useful  comments on an earlier 
  version of the paper.
  
\tableofcontents 

\section{\lie2s}{\label{S:lie2}}

In this section we review some basic facts about \lie2s. 
We follow the notations of  \cite{BaezCrans} (also see \cite{Roytenberg}).
All modules are over a fixed base commutative unital ring $K$.

 \begin{defn}{\label{D:lie2alg}}
  A {\bf \lie2} $\bbV$ consists of a 
  linear map $\partial \: V_1 \to V_0$ of  modules together with the 
  following data:
   \begin{itemize}
      \item three bilinear maps $[\cdot,\cdot] \: V_i\times V_j \to V_{i+j}$, $i+j=0,1$;
      \item  an antisymmetric trilinear map (the {\em Jacobiator}) 
         $\langle\cdot,\cdot,\cdot\rangle \:   V_0\times V_0\times V_0 \to V_1$.     
   \end{itemize}
  These maps satisfy the following axioms for all $w,x,y,z \in V_0$ and $h,k\in V_1$:
   \begin{itemize}
      \item $[x,y]=-[y,x]$;
      \item $[x,h]=-[h,x]$;
      \item $\partial([x,h])=[x,\partial h]$;
      \item $[\partial h, k]=[h,\partial k]$;
      \item $\partial \langle x,y,z\rangle = [x,[y,z]]+[y,[z,x]]+[z,[x,y]]$;
      \item  $\langle x,y,\partial h\rangle = [x,[y,h]]+[y,[h,x]]+[h,[x,y]]$;
      \item  $[\langle x,y,z\rangle,w]-[\langle w,x,y\rangle,z]+
      [\langle z,w,x\rangle,y]-[\langle y,z,w\rangle,x]=$ $$\langle [x,y],z,w\rangle + 
      \langle [z,w],x,y\rangle + \langle [x,z],w,y\rangle +
      \langle [w,y],x,z\rangle +\langle [x,w],y,z\rangle + \langle [y,z],x,w\rangle.$$
   \end{itemize}
  \end{defn}
 
We sometimes use the notation $[V_1 \to V_0]$ for a \lie2.

 \begin{defn}{\label{D:bracket}}
   The equality $[\partial h, k]=[h,\partial k]$ allows us to define  a bracket on
   $V_1$ by setting $[h,k]:=[\partial h, k]=[h,\partial k]$. 
 \end{defn}

 \begin{lem}{\label{L:failure}}
   For the bracket defined in Definition \ref{D:bracket}, the failure of 
   the Jacobi identity is measured by the equality
       $$\langle \partial h,\partial k,\partial h\rangle = [h,[k,l]]+[k,[l,h]]+[l,[h,k]].$$
 \end{lem}

 \begin{proof}
  Easy.
 \end{proof}

A  crossed-module in Lie algebras is the same thing as a  strict
\lie2, i.e., one for which the Jacobiator $\langle\cdot,\cdot,\cdot\rangle$ is identically 
zero. More precisely, given a \lie2 $\bbV$  with zero Jacobiator we obtain, by Lemma
\ref{L:failure}, a 
Lie algebra structure on $V_1$, where the bracket is
 as in Definition \ref{D:bracket}. This makes $\partial$ a 
Lie algebra homomorphism. The action of $V_0$ on $V_1$ is the given bracket $[\cdot,\cdot]\:
V_0\times V_1 \to V_1$. Also, observe that a strict \lie2 is the same thing as a 2-term dgla.

 \begin{defn}{\label{D:H}}
   Let $\bbV=[\partial \: V_1 \to V_0]$ be a \lie2. We define 
    $$H_1(\bbV):=\ker\partial, \ \ H_0(\bbV):=\coker\partial.$$
 \end{defn}

Note that $H_0(\bbV)$ and $H_1(\bbV)$ both inherit natural Lie algebra structures, the latter 
being necessarily abelian. Furthermore, $H_1(\bbV)$ is naturally an $H_0(\bbV)$-module.

 \begin{defn}{\label{D:morphismlie2alg}}
   A {\bf morphism} $f \: \bbW \to \bbV$ of \lie2s consists of the following data:
    \begin{itemize}
       \item linear maps $f_i \: W_i \to V_i$, $i=0,1$, commuting with the differentials;
       \item an antisymmetric bilinear map $\varepsilon \: W_0\times W_0 \to V_1$.
    \end{itemize} 
   These maps satisfy the following axioms:
    \begin{itemize}
      \item for every $x,y \in W_0$, $[f_0(x),f_0(y)]-f_0[x,y]=\partial\varepsilon(x,y)$;
      \item for every $x \in W_0$ and 
       $h \in W_1$, $[f_0(x),f_1(k)]-f_1[x,k]=\varepsilon(x,\partial k)$;
      \item for every $x,y,z \in W_0$, 
         {\small 
            $$\langle f_0(x),f_0(y),f_0(z)\rangle-f_1(\langle x,y,z\rangle)=$$
            $$\varepsilon(x,[y,z])+\varepsilon(y,[z,x])+\varepsilon(z,[x,y])+
                    [f_0(x),\varepsilon(y,z)]+ [f_0(y),\varepsilon(z,x)]+ [f_0(z),\varepsilon(x,y)].$$
                    }
     \end{itemize}
 \end{defn}


A morphism $f \: \bbW \to \bbV$ of \lie2s induces a Lie algebra homomorphism
$H_0(f) \: H_0(\bbW) \to H_0(\bbV)$ and an $H_0(f)$-equivariant morphism of Lie algebra 
modules $H_1(f) \: H_1(\bbW) \to H_1(\bbV)$.

\begin{defn}
  A morphism $f \: \bbW \to \bbV$ of \lie2s  is called an {\bf equivalence} 
  (or a {\bf quasi-isomorphism}) if $H_0(f)$ and $H_1(f)$ are isomorphisms.
\end{defn}

 \begin{defn}
  A morphism of \lie2 is {\bf strict} if $\varepsilon$ is identically zero. 
  In the case where $\bbV$ and $\bbW$ are crossed-modules in Lie algebras, 
  this means that $f$ is a (strict) morphism of crossed-modules.
 \end{defn}

 \begin{defn}{\label{D:compositionmorph}}
    If $f=(f_0,f_1,\varepsilon)\: \bbW \to \bbV$ and $g=(g_0,g_1,\delta) \: \bbV \to \bbU$
    are morphisms of \lie2s, the composition $gf$ is defined to be the triple 
    $(g_0f_0,g_1f_1,\gamma)$, where
     $$\gamma(x,y):=g_1\varepsilon(x,y)+\delta(f_0(x),f_0(y)), \ \ x,y \in W_0.$$
 \end{defn}

Finally, we recall the definition of a transformation between morphisms of \lie2s. Up to a minor
difference in sign conventions, it is the same as \cite{Roytenberg}, Definition 2.20.
It is also equivalent to Definition
37 in the archive version [arXiv:math/0307263v5] of \cite{BaezCrans}.

 \begin{defn}{\label{D:transformation}}
   Given morphisms $f,g \: \bbW \to \bbV$ of \lie2s, a {\bf transformation} 
   (or  an {\bf  $L_{\infty}$-homotopy})
   from $g$ to $f$ is a linear map $\theta \: W_0 \to V_1$ such that
   \begin{itemize}
      \item for every $x \in W_0$, $f_0(x)-g_0(x)=\partial\theta(x)$;
      \item for every $h \in W_1$, $f_1(h)-g_1(h)=\theta(\partial h)$;
       \item for every $x,y \in W_0$, 
     {\small $$[\theta(x),\theta(y)]-\theta([x,y])=\varepsilon_f(x,y)-\varepsilon_g(x,y)+
       [g_0(y),\theta(x)]+[\theta(y),g_0(x)].$$}
   \end{itemize} 
 \end{defn}
 
 \begin{rem}
    It may perhaps look more natural to consider such a $\theta$ as a transformation from
    $f$ to $g$ and not from $g$ to $f$. (This is how it is in \cite{Roytenberg}, Definition 2.20.)
    We, however, choose the backward convention to be compatible with corresponding notion
    of transformation for butterflies (Definition \ref{D:butterflymorphism}).
 \end{rem}
 
 It is easy to see that if $f$ and $g$ are related by a transformation, 
 then $H_i(f)=H_i(g)$, $i=0,1$.
 
 \begin{defn}{\label{D:compositiontrans}}
   If $\theta$ is a transformation from $f$ to $g$ and $\sigma$ a transformation from $g$ to
   $h$, their composition is the transformation from $f$ to $h$ given by the linear map
    $\theta+\sigma$.
 \end{defn}
 
The following definition is the one in \cite{BaezCrans}, Proposition 4.3.8.
 
 \begin{defn}{\label{D:2term}}
  We define $\TwoTerm$ to be the 2-category in which the objects are \lie2s, the morphisms
  are as in Definition \ref{D:morphismlie2alg} and the 2-morphisms are as in 
  Definition \ref{D:transformation}.
 \end{defn}
\section{Butterflies between \lie2s}{\label{S:Butterfly}}

In this section we introduce the notion of a butterfly between  \lie2s and show that
butterflies encode morphisms of \lie2s (Propositions \ref{P:morphism}, \ref{P:transformation}).
A butterfly should be regraded as an analogue of a Morita morphism.

 \begin{defn}{\label{D:butterfly}}
  Let $\bbV$ and $\bbW$ be \lie2s. A {\bf butterfly} $B \: \bbW \to \bbV$ is a 
  commutative diagram
      $$\xymatrix@C=8pt@R=6pt@M=6pt{ W_1 \ar[rd]^{\kappa} \ar[dd]
                          & & V_1 \ar[ld]_{\iota} \ar[dd] \\
                            & E \ar[ld]^{\sigma} \ar[rd]_{\rho}  & \\
                                       W_0 & & V_0       }$$
  of modules in which $E$ is endowed with an antisymmetric bracket $[\cdot,\cdot] \:
  E\times E \to E$ satisfying the following axioms:
    \begin{itemize}
      \item both diagonal sequences are complexes and the NE-SW sequence
         $$0\to V_1 \llra{\iota} E \llra{\sigma} W_0 \to 0$$
       is short exact;
      \item  for every $a,b \in E$, 
         $$\rho[a,b]=[\rho(a),\rho(b)] \ \ \text{and}  \ \ \sigma[a,b]=[\sigma(a),\sigma(b)];$$   
      \item for every
        $a \in E$, $h \in V_1$, $l \in W_1$,
            $$[a,\iota(h)]=\iota[\rho(a),h] \ \ \text{and} \ \ [a,\kappa(l)]=\kappa[\sigma(a),l];$$ 
      \item  for every $a,b,c \in E$,  
        {\small $$\iota\langle\rho(a),\rho(b),\rho(c)\rangle+
           \kappa\langle\sigma(a),\sigma(b),\sigma(c)\rangle=[a,[b,c]]+[b,[c,a]]+[c,[a,b]].$$}
    \end{itemize}
 \end{defn}    

In the case where $\bbV$ and $\bbW$ are crossed-modules in Lie algebras (i.e., when the
 Jacobiators are identically zero), the bracket on $E$ makes it into a Lie algebra and all 
the maps in  the butterfly diagram become Lie algebra homomorphisms.

 \begin{rem}{\label{R:cone}}
   The map $\kappa+\iota \: W_1\oplus V_1 \to E$ has a natural \lie2 structure. Let us denote
   this \lie2 by $\bbE$. The two projections $\bbE \to \bbW$ and $\bbE \to \bbV$ are strict
   morphisms of \lie2s and the former is a quasi-isomorphism. Thus, we can think of
   the butterfly $B$ as a zig-zag of strict morphisms from $\bbW$ to $\bbV$. 
 \end{rem}

 \begin{defn}{\label{D:butterflymorphism}}
   Given two butterflies $B,B' \: \bbW \to \bbV$, a {\bf morphism}  of butterflies from $B$ to $B'$
   is a linear map $E \to E'$ commuting with the brackets and all four structure maps  
   of the butterfly. (Note that such a map $E \to E'$ is necessarily an isomorphism.)   
 \end{defn}

A butterfly $B \: \bbW \to \bbV$ induces a Lie algebra homomorphism
$H_0(B) \: H_0(\bbW) \to H_0(\bbV)$ and an $H_0(B)$-equivariant morphism 
$H_1(B) \: H_1(\bbW) \to H_1(\bbV)$. If $B$ and $B'$ are related by a morphism, 
then $H_i(B)=H_i(B')$,  $i=0,1$.

Let $f \: \bbW \to \bbV$ be a morphism of \lie2s as in Definition 
\ref{D:morphismlie2alg}. Define a bracket on $V_1\oplus W_0$ by the rule
  $$[(k,x),(l,y)]:=\big([k,l]+[f_0(x),l]+[k,f_0(y)]+\varepsilon(x,y),[x,y]\big).$$
Define the following four maps:
  \begin{itemize}
      \item $\kappa \: W_1 \to V_1\oplus W_0$, $\kappa(l)=(-f_1(l),\partial l)$,
      \item $\iota \: V_1 \to V_1\oplus W_0$, $\iota(k)=(k,0)$,
      \item $\sigma \: V_1\oplus W_0 \to W_0$, $\sigma(k,x)=x$,
      \item $\rho \: V_1\oplus W_0 \to V_0$, $\rho(k,x)=\partial k+f_0(x)$.
  \end{itemize}

 \begin{prop}{\label{P:morphism}}
    With the bracket on $V_1\oplus W_0$ and the maps 
    $\kappa$, $\iota$, $\rho$ and $\sigma$ defined as above,
    the diagram 
            $$\xymatrix@C=8pt@R=6pt@M=6pt{ W_1 \ar[rd]^{\kappa} \ar[dd]
                          & & V_1 \ar[ld]_{\iota} \ar[dd] \\
                            & V_1\oplus W_0 \ar[ld]^{\sigma} \ar[rd]_{\rho}  & \\
                                       W_0 & & V_0       }$$
     is a butterfly (Definition \ref{D:butterfly}). Conversely, given a butterfly as in  
     Definition \ref{D:butterfly} and a linear section $s \: W_0 \to E$ to $\sigma$,
     we obtain a morphism of \lie2s by setting 
     $$f_0:=\rho s, \ \ f_1:=s \partial-\kappa, \ \ \varepsilon:=[s(\cdot),s(\cdot)]-
        s[\cdot,\cdot].$$   
    (In the definition of the last two maps we are using the exactness of the 
    NE-SW sequence.)  Furthermore, these two constructions are inverse to each other.                         
 \end{prop}

 \begin{prop}{\label{P:transformation}}
      Via the construction introduced in Proposition \ref{P:morphism}, transformations 
      between morphisms of \lie2s (Definition \ref{D:transformation})
      correspond to morphisms of butterflies (Definition \ref{D:butterflymorphism}). 
      In other words,  we have an equivalence of groupoids between the groupoid of 
      morphisms of \lie2s from $\bbW$ to $\bbV$ and the groupoid of butterflies 
      from $\bbW$ to $\bbV$.
 \end{prop}

 \begin{ex} Let $V$ and $W$ be Lie algebras. Define $\mathbb{D}er(V)$ to 
  be the crossed-module in  Lie algebras $\partial \: V\to \operatorname{Der}(V)$, 
  where $\partial$ sends $v \in V$ to
  the derivation $[v,\cdot]$. Then, the equivalence classes of \lie2 morphisms
  $W \to \mathbb{D}er(V)$ are in bijection with isomorphism classes of extensions of 
  $W$ by $V$. Here, $W$ is regarded as the \lie2 $[0\to W]$.
 \end{ex}

\section{Homotopy fiber of a morphism of \lie2s}{\label{S:Cone}}

We introduce the homotopy fiber (or ``shifted mapping cone'') of a butterfly (and also of 
a morphismof \lie2s). The homologies of the homotopy fiber sit in a 7-term long 
exact sequence. We see in $\S$ \ref{SS:structure} that the homotopy fiber has a 
rich structure consisting of various brackets.

 \begin{defn}{\label{D:cone}}
  Let $B \: \bbW \to \bbV$, 
      $$\xymatrix@C=8pt@R=6pt@M=6pt{ W_1 \ar[rd]^{\kappa} \ar[dd]
                          & & V_1 \ar[ld]_{\iota} \ar[dd] \\
                            & E \ar[ld]^{\sigma} \ar[rd]_{\rho}  & \\
                                       W_0 & & V_0       }$$
  be a butterfly. We define the {\bf homotopy fiber} $\Hfib(B)$ of $B$ to be the NW-SE sequence
                 $$W_1 \llra{\kappa} E \llra{\rho} V_0.$$
  We will think of $W_1$, $E$ and $V_0$ as sitting in degrees $1$,$0$ and $-1$.  
 \end{defn}
 
The homotopy fiber measures the deviation of $B$ from being an equivalence
(see Remark \ref{R:mappingcone} below). 

 \begin{prop}{\label{P:long}} 
  More precisely,  we have a long exact sequence
   {\small $$\xymatrix@C=14pt@R=11pt@M=4pt{ 0 \ar[r] & H_1(\Hfib(B))\ar[r]
                 & H_1(\bbW) \ar[r]^{H_1(B)} & H_1(\bbV) \ar[r]
                 & H_0(\Hfib(B))  \ar `/7pt[d] `[l] `[dll] `[r] [dl]  &  & &\\
                 & & & H_0(\bbW)\ar[r]^{H_0(B)} 
                 & H_0(\bbV)\ar[r] & H_{-1}(\Hfib(B)) \ar[r] & 0. }$$}
 \end{prop}

 \begin{proof}
  Exercise.
 \end{proof}                 
                 
Except for $H_{-1}(\Hfib(B))$, all the terms in the above sequence are Lie algebras
and all the maps are Lie algebra homomorphisms; see $\S$\ref{SS:structure}  below. 

 \begin{cor}{\label{C:flip}}
  A butterfly $B$ is an equivalence (i.e.,  induces isomorphisms on $H_0$ and 
  $H_1$) if and only if its NW-SE sequence is short exact. In this case, the inverse of
  $B$ is obtained by flipping it along the vertical axis.
 \end{cor}

Definition \ref{D:cone} leads to the following definition. 

 \begin{defn}
   For a  morphism $f=(f_0,f_1,\varepsilon)\: \bbW \to \bbV$ of \lie2s, 
   we define its homotopy fiber $\Hfib(f)$  to be the sequence
   $$W_1 \llra{(-f_1,\partial)} V_1\oplus W_0 \llra{\partial+f_0} V_0.$$
 \end{defn}
   
 \begin{rem}{\label{R:mappingcone}}
   If we forget all the brackets and index the terms of $\Hfib(f)$ by $2$,$1$,$0$, we see that
   $\Hfib(f)$ coincides with the cone of $f$ in the derived category of chain complexes. 
 \end{rem}
 
\subsection{Structure of the homotopy fiber}{\label{SS:structure} 

The $\Hfib(B)$ comes with some additional structure which we discuss below. 
First, let us rename the homotopy fiber in the following way
   $$C_1 \llra{\partial} C_0 \llra{\partial} C_{-1}.$$
We have the following data:
 \begin{itemize}
    \item antisymmetric bilinear  brackets $[\cdot,\cdot]_i \: C_i\times C_i \to C_i$, $i=1,0,-1$;
    \item antisymmetric bilinear brackets $[\cdot,\cdot]_{01} \: C_0\times C_1 \to C_1$, 
    $[\cdot,\cdot]_{10} \: C_1\times C_0 \to C_1$;
    \item antisymmetric trilinear Jacobiators $\langle\cdot,\cdot,\cdot\rangle_i \: 
    C_i \times C_i \times C_i \to C_{i+1}$, $i=-1,0$.
 \end{itemize}
We denote $[\cdot,\cdot]_{-1}$ by $[\cdot,\cdot]$. 
The following axioms are satisfied:
 \begin{itemize}
    \item $[\cdot,\cdot]_{01}=-[\cdot,\cdot]_{10}$;
    \item for every $a \in C_0$ and $h \in C_1$, $\partial([a,h]_{01})=[a,\partial h]_{0}$;
    \item for every $h,k \in C_1$, $[h,k]_{1}=[\partial h, k]_{01}=[h,\partial k]_{10}$;
    \item for every $a,b \in C_0$, $\partial[a,b]_0=[\partial a,\partial b]$;
    \item   for every $a,b,c \in C_0$,  
        {\small $$\langle\partial a,\partial b,\partial c\rangle_{-1}+
           \partial(\langle a,b,c\rangle_0)=
              [a,[b,c]_0]_0+[b,[c,a]_0]_0+[c,[a,b]_0]_0.$$}
    \item for every $a,b \in C_0$ and $h \in C_1$,  
       $$\langle a,b,\partial h\rangle_0=
              [a,[b, h]_{01}]_{01}+[b,[h,a]_{10}]_{01}+[h,[a,b]_0]_{10}.$$
    \item for every $a,b,c,d \in C_0$, 
    {\small $$[\langle a,b,c\rangle_0,d]_{10}-[\langle d,a,b\rangle_0,c]_{10}+
    [\langle c,d,a\rangle_0,b]_{10}-[\langle b,c,d\rangle_0,a]_{10}=$$ 
    $$\langle [a,b]_0,c,d\rangle_0 + \langle [c,d]_0,a,b\rangle_0 + \langle [a,c]_0,d,b\rangle_0 +
      \langle [d,b]_0,a,c\rangle_0 +\langle [a,d]_0,b,c\rangle_0 + \langle [b,c]_0,a,d\rangle_0.$$}
 \end{itemize}

There are natural chain maps $\bbV \to \Hfib(B)[-1]$ and $\Hfib(B) \to \bbW$ which respect all
the brackets on the nose. In fact, $\Hfib(B) \to \bbW \to \bbV$ is an exact triangle in the derived
category of chain complexes (note the reverse shift due to homological indexing).
 
\section{The bicategory of Lie 2-algebras and butterflies}{\label{S:Bicat}}
 
Given butterflies
  $$\xymatrix@C=8pt@R=6pt@M=6pt{ W_1 \ar[rd]^{\kappa}  \ar[dd] 
                            & & V_1 \ar[ld]_{\iota}  \ar[dd]  \\
                  & E \ar[ld]^{\sigma}   \ar[rd]_{\rho}   & \\
                  W_0 & & V_0       } \ \ \ \ \
 \xymatrix@C=8pt@R=6pt@M=6pt{ V_1 \ar[rd]^{\kappa'}  \ar[dd] 
                            & & U_1 \ar[ld]_{\iota'}  \ar[dd]  \\
                  & F \ar[ld]^{\sigma'}  \ar[rd]_{\rho'}  & \\
                  V_0 & & U_0       }$$
we define their composition to be the butterfly
    $$\xymatrix@C=20pt@R=10pt@M=4pt{ W_1 \ar[rd]^{(\kappa,0)}  \ar[dd] 
                            & & U_1 \ar[ld]_{(0,\iota')}   \ar[dd]  \\
                  & E\ouplus{V_0}{V_1}F \ar[ld]^{\sigma\circ\pr}   \ar[rd]_{\rho'\circ\pr}   & \\
                  W_0 & & U_0       }$$
Here $E\ouplus{V_0}{V_1}F$ is, by definition, the fiber product of
$E$ and $F$ over $V_0$ modulo the diagonal image of $V_1$ via $(\iota,\kappa')$.
The bracket on it is defined component-wise.

 \begin{prop}
    With butterflies as morphisms, morphisms of butterflies as 
    $2$-morphisms, and composition defined as above, \lie2s form a bicategory $\TwoTermB$.
 \end{prop}

For a \lie2 $\bbV$, the identity butterfly from $\bbV$ to itself is defined to be
              $$\xymatrix@C=8pt@R=6pt@M=6pt{ V_1 \ar[rd]^{\kappa} \ar[dd]
                          & & V_1 \ar[ld]_{\iota} \ar[dd] \\
                            & V_1\oplus V_0 \ar[ld]^{\sigma} \ar[rd]_{\rho}  & \\
                                       V_0 & & V_0       }$$
Here, the bracket on $V_1\oplus V_0$ is defined by
                        $$[(k,x),(l,y)]:=\big([k,l]+[x,l]+[k,y],[x,y]\big).$$
The four structure maps of the butterfly are:
  \begin{itemize}
      \item $\kappa \: V_1 \to V_1\oplus V_0$, $\kappa(l)=(-l,\partial l)$
      \item $\iota \: V_1 \to V_1\oplus V_0$, $\iota(k)=(k,0)$
      \item $\sigma \: V_1\oplus V_0 \to V_0$, $\sigma(k,x)=x$
      \item $\rho \: V_1\oplus V_0 \to V_0$, $\rho(k,x)=\partial k+x$
  \end{itemize}

 \begin{prop}{\label{P:biequiv}}
    The construction of Proposition \ref{P:morphism} induces a biequivalence 
    $\TwoTerm\cong\TwoTermB$.
 \end{prop}

 \begin{proof}
  Straightforward verification.
 \end{proof}

  By Lemma \ref{C:flip}, a butterfly $B \: \bbW \to \bbV$ is invertible (in the bicategorical sense) 
  if and only if its NW-SE sequence is also short exact. In this case, the inverse of 
  $B$ is obtained by flipping $B$ along the vertical axis.

\subsection{Composition of a  butterfly with a strict morphism}{\label{SS:strictcompose}}

Composition of butterflies takes a simpler from when one of the butterflies comes
from a strict morphism. When the first morphisms is strict, say
  $$\xymatrix@C=8pt@R=6pt@M=6pt{ W_{1}  \ar[rr]^{f_{1}} \ar[dd]
                              & & V_{1}   \ar[dd]  \\
                                           &     & \\
                        W_0 \ar[rr]_{f_0} & & V_0       }$$
then the composition is
      $$   \xymatrix@C=8pt@R=6pt@M=6pt{ W_{1} \ar[rd]   \ar[dd]
                            & & U_{1} \ar[ld]  \ar[dd]  \\
      & f^{*}_0(F) \ar[ld]^{f^{*}_1(\sigma')} \ar[rd] & \\
                         W_0 & & U_0       }$$
Here, $f^{*}_0(F)$ stands for the pullback of the extension $F$
along $f_0 \: W_0 \to V_0$. More precisely,
$f^{*}_0(F)=W_0\oplus_{V_0}F$ is the fiber product.

When the second morphisms is strict, say
          $$\xymatrix@C=8pt@R=6pt@M=6pt{ V_{1}  \ar[rr]^{g_{1}}
          \ar[dd]
                            & & U_{1}   \ar[dd]  \\
                  &          & \\
                  V_0 \ar[rr]_{g_0} & & U_0       }$$
 then the composition is
         $$   \xymatrix@C=8pt@R=6pt@M=6pt{ W_{1} \ar[rd]   \ar[dd]
                    & & U_{1} \ar[ld]_{g_{1,*}(\iota)}  \ar[dd]  \\
                                & g_{1,*}(E) \ar[ld]  \ar[rd]   & \\
                                              W_0 & & U_0      }$$
Here, $g_{1,*}(E)$ stands for the push forward of the extension $E$
along $g_{1} \: V_{1} \to U_{1}$. More precisely,
$g_{1,*}(E)=E\oplus_{V_{1}} U_{1}$ is the pushout.

\section{Weak morphisms of Lie 2-groups and butterflies}{\label{S:Weak}}

There are at least three equivalent ways to define weak
morphisms of Lie 2-groups. One way is to localize the 2-category of Lie 2-groups and
strict morphisms with respect to equivalences, and define weak morphisms to be 
morphisms in this localized category
(by definition, an equivalence between Lie 2-groups is a 
morphism  which induces isomorphisms on $\pi_0$ and $\pi_1$).

The second definition is that a weak morphism of Lie 2-groups is 
a weak morphism (i.e., a monoidal functor) between the associated Lie group stacks. 
 
The third definition, which is shown
in \cite{ButterflyI} to be equivalent to the stack definition, makes use of butterflies and is the
subject of this section. 
It is the butterfly definition that proves to 
be most suitable for the study of connected covers
of Lie 2-groups and also for proving our integration result (Theorem \ref{T:main}).

\subsection{A note on terminology}{\label{SS:terminology}
Lie 2-group could mean  different things to different people, so  some clarification
in terminology is in order before we move on. One definition is that a Lie 2-group is
a crossed-module $\bbG:=[\partial \: G_1 \to G_0]$ in the category of Lie groups. Although
most known examples of Lie 2-groups are of this  form, this is not the most general definition,
as it is too {\em strict}.

Arguably, the correct definition is that  a Lie 2-group is a 
differentiable group stacks, that is, a (weak)
group object  $\mathcal{G}$ in the category of differentiable stacks. Every Lie crossed-module
$[G_1 \to G_0]$ gives rise to a group stack $[G_0/G_1]$, but not every differentiable group stack
is of this form. A  differentiable group stack $\mathcal{G}$ comes from a Lie crossed-module
if and only if it admits an atlas $\varphi \: G_0 \to \mathcal{G}$ such that $G_0$ is a Lie group
and $\varphi$ is a differentiable  (weak) homomorphism.

Another definition of a Lie 2-group (which is presumably equivalent to the stack definition) is
discussed in the Appendix of \cite{Henriques}. This definition is motivated by the fact that a Lie
2-group gives rise to a simplicial manifold and, conversely, a simplicial manifold with
certain fibrancy properties and some conditions on its homotopy
groups should come from a Lie 2-group.

In this paper we only deal with strict Lie 2-groups, i.e., 
Lie group stacks coming from  Lie crossed-modules.
Presumably our theory can be extended to arbitrary Lie 2-groups.

Throughout the text, all Lie groups are assumed to be finite dimensional unless otherwise stated.

\subsection{Quick review of Lie butterflies}{\label{SS:Liebutterfly}}

For more details on butterflies see \cite{Maps}, especially $\S$9.6, $\S$10.1, and \cite{ButterflyI}. 
In what follows, by a homomorphism of Lie groups we mean a differentiable homomorphism.

 \begin{rem}
  In \cite{Maps} and \cite{ButterflyI} we use the right-action convention for crossed-modules, while
  in these notes, in order to be compatible with the existing literature on $L_{\infty}$-algebras, 
  we have used the  left-action convention for Lie algebra crossed-modules. Therefore, 
  for the sake of consistency,   we will adopt the   left-action convention for Lie 
  crossed-modules as well.
 \end{rem}

Let $\bbG$ and $\bbH$ be a Lie crossed-modules (i.e., a crossed-modules 
in the category of Lie groups). A {\bf butterfly} $B \: \bbH \to \bbG$ is a commutative diagram
      $$\xymatrix@C=8pt@R=6pt@M=6pt{ H_1 \ar[rd]^{\kappa} \ar[dd]
                          & & G_1 \ar[ld]_{\iota} \ar[dd] \\
                            & E \ar[ld]^{\sigma} \ar[rd]_{\rho}  & \\
                                       H_0 & & G_0       }$$
in which both diagonal sequences are complexes of Lie groups, and the NE-SW
sequence is short exact. We also require that
for every $x \in E$, $\alpha \in G_2$ and $\beta \in H_2$ the following equalities hold:
    $$\iota(\rho(x)\cdot\alpha)=x\iota(\alpha) x^{-1}, \ \
      \kappa(\sigma(x)\cdot\beta)=x\kappa(\beta) x^{-1}.$$

A butterfly between Lie crossed-modules can be regarded as a Morita morphism which respects
the group structures. A morphism $B \to B'$ of butterflies is, by definition,  
a homomorphism $E\to E'$ of 
Lie groups which commutes with all four structure maps of the butterflies. Note that 
such a morphism is necessarily an isomorphism. 

 \begin{rem}
  For the reader interested in the  topological version of the story, we remark that in the
  definition of a topological butterfly one needs to assume that the map $\sigma \:
  E \to H_0$, viewed as a continuous map of topological spaces, admits local sections.
  This is automatic in the Lie case because $\sigma$ is a submersion.
 \end{rem}

Thus, with butterflies as morphisms, Lie crossed-modules form a bicategory in 
which every $2$-morphism is an isomorphism. We denote this bicategory by $\LieXM$. 
The following theorem justifies why butterflies provide the right notion of morphism.

 \begin{thm}[\oldcite{ButterflyI}]
  The 2-category of (strict) Lie 2-groups  and weak morphisms is biequivalent
  to the bicategory  $\LieXM$ of Lie crossed-modules and butterflies.
 \end{thm}

We recall (\cite{Maps}, $\S$10.1) how composition of two butterflies  $C \: \bbK \to \bbH$ and 
$B \: \bbH \to \bbG$ is defined. Let $F$ and $E$ be the Lie groups appearing in the center 
of these butterflies, respectively. Then, the composition $B\circ C$ is  the butterfly
    $$\xymatrix@C=10pt@R=4pt@M=4pt{ K_1 \ar[rd]  \ar[dd]
                            & & G_1 \ar[ld]  \ar[dd] \\
                  & F\oux{H_0}{H_1}E \ar[ld]  \ar[rd]   & \\
                  K_0 & & G_0       }$$
where $F\oux{H_0}{H_1}E$ is the fiber product $F\oux{H_0}{}E$ modulo the diagonal image 
of $H_1$.

In the case where one of the butterflies is strict, the composition takes a simpler form similar to 
the discussion of $\S$\ref{SS:strictcompose}. See (\cite{Maps}, $\S$10.2) for more details.

\section{Connected covers of a Lie 2-group}{\label{S:Connected}}

In this section  we construct  $n$-th connected covers $\bbG\ang{n}$ of a Lie crossed-module
$\bbG=[G_1\to G_0]$
for $n=0,1,2$. In $\S$\ref{S:Functoriality} we prove that
these are functorial with respect to butterflies. Hence,  in particular, they are
invariant under equivalence of Lie crossed-modules (Corollary \ref{C:cover}). All Lie groups
are assumed to be finite dimensional unless otherwise stated.

\medskip
\noindent {\em Caveat on notation.} In this section,  by
the $i$-th homotopy group $\pi_n\bbG$ of a 
topological crossed-module $\bbG=[\partial\: G_1 \to G_0]$ we mean the $i$-th homotopy 
of the simplicial space associated to it (equivalently, the $i$-th homotopy 
of the quotient stack $G_0/G_1]$). When $i=0,1$, this should not be confused with the
usage of $\pi_0$ and $\pi_1$ for $\coker\partial$ and $\ker\partial$; the two notations
agree only when $G_0$ and $G_1$ are discrete groups. 

\medskip
 
Recall that a map $f \: X \to Y$ of topological spaces is $n$-connected
if $\pi_i f \: \pi_i X \to \pi_i Y$ is an isomorphism of $i \leq n$ and  a surjection for $i=n+1$. 
  
 \begin{prop}{\label{P:connected}} Let $\bbG=[G_1\to G_0]$ be a topological crossed-module
  and $n\geq 0$ an integer.  The following  are equivalent:
    \begin{itemize}
      \item[i)] The map $\partial$ is $(n-1)$-connected; 
      \item[ii)] The quotient stack $\mathcal{G}:=[G_0/G_1]$ is $n$-connected (in the sense
         of \cite{Foundations}, $\S$17);
      \item[iii)] The classifying space of $\mathcal{G}$ is $n$-connected. (We are
      viewing $\mathcal{G}$ as   a  stack and ignoring its group structure.)
    \end{itemize}
 \end{prop}
 
 \begin{proof}
  The equivalence of  (i) and (ii) follows from the homotopy fiber sequence
  applied to the fibration of stacks $G_1\to G_0 \to [G_0/G_1]$.
  The equivalence of (ii) and (iii) follows from \cite{Homotopytype}, Theorem 10.5.
 \end{proof}

 \begin{defn}{\label{D:connected}}
   We say that a Lie crossed-module $\bbG$ is {\bf $n$-connected} if it satisfies
   the equivalent conditions of Proposition \ref{P:connected}.
 \end{defn}

It follows from Proposition \ref{P:connected} that the notion of $n$-connected is
invariant under equivalence of Lie crossed-modules.
 
 \begin{rem} A   Lie crossed-module $\bbG$ is
   $2$-connected if and only if $\pi_i\partial \: \pi_iG_1 \to \pi_iG_0$
   is an isomorphism for $i=0,1$. This is because   $\pi_2$ of every (finite dimensional) 
   Lie group vanishes.
 \end{rem}
 
\subsection{Definition of the connected covers}{\label{SS:definition}}

In this subsection we define the $n$-th connected cover of a Lie crossed-module
for $n\leq 2$. In the
next section we prove that these definitions are functorial with respect to butterflies.
In particular, it follows that they are invariant under equivalence of Lie crossed-modules.
The discussion of this and the next section is valid for topological 
crossed-modules (and also for infinite-dimensional Lie crossed-modules) as well.

\medskip

\noindent{\bf The $0$-th connected cover of $\bbG$.}  
It is easy to see that a Lie 2-group $\mathcal{G}$
is connected if and only if it has a presentation by a Lie crossed-module $[G_1 \to G_0]$ 
with $G_0$ connected. (Proof: choose an atlas $\varphi \: G_0 \to \mathcal{G}$ such that
$\varphi$ is a homomorphism and $G_0$ is connected, and set 
$G_1:=\{1_{\mathcal{G}}\}\times_{\mathcal{G},\varphi}G_0$.)
For a given Lie crossed-module  $\bbG=[G_1 \to G_0]$
its $0$-th connected cover is defined to be
      $$\bbG\ang{0}:=[\partial^{-1}(G_0^o) \to G_0^o],$$
where  $G^o$ stands for the connected component of the identity.
The crossed-module $\bbG\ang{0}$ should be thought of as the connected component
of the identity of $\bbG$. There is an obvious strict morphism 
$q_0 \: \bbG\ang{0} \to \bbG$ which induces  isomorphisms on $\pi_i$ for $i\geq 1$
(see Proposition \ref{P:cover}).
 
\medskip
\noindent{\bf The $1$-st connected cover of $\bbG$.} 
A Lie 2-group $\mathcal{G}$
is $1$-connected if and only if it has a presentation by a Lie crossed-module $[G_1 \to G_0]$ with
$G_0$ 1-connected and $G_1$ connected.
(Proof: choose an atlas $\varphi \: G_0 \to \mathcal{G}$ such that
$\varphi$ is a homomorphism and $G_0$ is 1-connected, and set 
$G_1:=\{1_{\mathcal{G}}\}\times_{\mathcal{G},\varphi}G_0$.)
For a given Lie crossed-module  $\bbG=[G_1 \to G_0]$
its $1$-st connected cover is defined to be
         $$\bbG\ang{1}:=[L^o \to  \widetilde{G_0^o}],$$ 
where $L := G_1\times_{G_0}\widetilde{G_0^o}$ and \,$\tilde{}$\, stands for universal cover.
There is an obvious strict morphism $q_1 \: \bbG\ang{1} \to \bbG$ which
factors through $q_0$ and  induces  isomorphisms on $\pi_i$ for $i\geq 2$
(see Proposition \ref{P:cover}).

\medskip
\noindent{\bf The $2$-nd connected cover of $\bbG$.}
A Lie 2-group  $\mathcal{G}$
is $2$-connected if and only if it has a presentation by a Lie crossed-module $[G_1 \to G_0]$ with
$G_0$ and $G_1$  both 1-connected.
(Proof: choose an atlas $\varphi \: G_0 \to \mathcal{G}$ such that
$\varphi$ is a homomorphism and $G_0$ is 1-connected, and set 
$G_1:=\{1_{\mathcal{G}}\}\times_{\mathcal{G},\varphi}G_0$.)
For a given Lie crossed-module  $\bbG=[G_1 \to G_0]$
its $2$-nd connected cover is defined to be
      $$\bbG\ang{2}:=[\widetilde{L^o} \to  \widetilde{G_0^o}],$$
where $L$ is as in the previous part. 
There is an obvious strict morphism $q_2 \: \bbG\ang{2} \to \bbG$ which
factors through $q_1$ and  induces  isomorphisms on $\pi_i$ for $i\geq 3$
(see Proposition \ref{P:cover}).

 \begin{rem}
  Note that a $1$-connected Lie group is automatically $2$-connected. The same is not
  true for Lie 2-groups.
 \end{rem}
 
\subsection{Uniform definition of the $n$-connected covers}{\label{SS:uniform}}

In order to be avoid repetition in the constructions and arguments given in the next section,
we phrase the definition of  $\bbG\ang{n}$ in a uniform manner for $n=0,1,2$, and single out the
main properties of the connected covers $q_n \: G\ang{n} \to G$ which will be needed
in the next section.\footnote{Apart from improving the clarity of proofs in the next 
section, there is another purpose for singling out properties of
connected covers in the form of axioms $\bigstar$: in contexts other than Lie crossed-modules,
it may be possible to arrange for the axioms $\bigstar$ for, say, other values of $n$, or 
by using different constructions for $G\ang{n}$. In such cases, our proofs apply verbatim.} 
Our discussion will be valid for topological crossed-modules (and also for infinite-dimensional 
Lie crossed-modules) as well.
 
First off, we need  functorial $n$-connected covers 
$q_n\: G\ang{n} \to G$ for $n=0,1,2$.\footnote{This, in fact, can be arranged 
for any $n$ in the category of  topological groups.}  
We set $G\ang{-1}=G$. We take $G\ang{0}:=G^o$ and
$G\ang{1}=G\ang{2}=\widetilde{G^o}$, where $G^o$ means connected component of the
identity. (In the case where $G$ is a topological group, or an infinite dimensional Lie
group, one has to make a different choice for $G\ang{2}$; see Remark \ref{R:contractible}.)

For a crossed-module
$\bbG=[G_1 \to G_0]$ we define $\bbG\ang{n}$ to be 
       $$\bbG\ang{n}:=[\partial \:L\ang{n-1} \to G_0\ang{n}],$$
where  $L :=G_1\times_{G_0,q_n}G_0\ang{n}$, and $\partial=\pr_2\circ q_{n-1}$.  
The action of $G_0\ang{n}$ on $L\ang{n-1}$
is defined as follows. There is an  action of $G_0\ang{n}$ on $L$ defined
componentwise (on the first
component it is obtained, via $q_n$, from the action of $G_0$ on $G_1$
and on  the second component it is given by right conjugation). By functoriality
of the $n$-th connected cover construction (applied to $L$), 
this action lifts to $L\ang{n-1}$. For $\bbG\ang{n}$ to be a crossed-module, we use the following
property:
  \begin{itemize}
   \item[($\bigstar$0)] For every $x \in  G\ang{n-1}$, the action of $q_{n-1}(x) \in G$ on 
    $G\ang{n-1}$ obtained (by functoriality) from  the conjugation action of $q_{n-1}(x)$ on $G$
    is equal to conjugation by $x$.
  \end{itemize}

There is a strict morphism of crossed-modules $q_n \: \bbG\ang{n} \to \bbG$
defined by
    $$\xymatrix@C=14pt@R=8pt@M=6pt{ L\ang{n-1}  \ar[rr]^(0.55){\pr_1 \circ q_{n-1}} 
      \ar[dd]_{\partial} 
                              & & G_{1}   \ar[dd]^{\partial} \\
                                           &     & \\
                       G_0\ang{n} \ar[rr]_{q_n} & & G_0       }$$
  
We will also need the following property:
  
\begin{itemize}
  \item[($\bigstar$1)] The map $q_{n-1} \: G\ang{n-1} \to G$  admits
  local sections near every point in its image (hence is a fibration with open-closed image).   
 \end{itemize} 

 \begin{prop}{\label{P:cover}} For $i\leq n$, we have $\pi_i(\bbG\ang{n})=\{0\}$.
  For $i \geq n+1$ the morphism $q_n \: \bbG\ang{n} \to \bbG$ induces isomorphisms 
  $\pi_i(q_n) \: \pi_i(\bbG\ang{n}) \to \pi_i(\bbG)$.
 \end{prop}

 \begin{proof}
  Consider the commutative diagram
      $$\xymatrix@C=22pt@R=22pt@M=6pt{
        L\ang{n-1} \ar[r]^{\partial}  \ar[d]_{\pr_1 \circ q_{n-1}} 
         & G_0\ang{n} \ar[r]  \ar[d]^{q_n} & 
              \bbG\ang{n}  \ar[d]^{q_n}  \\
        G_1 \ar[r]_{\partial} & G_0 \ar[r] & \bbG
      }$$
  Both rows are fibrations of crossed-modules (so, induce fibrations on the classifying spaces).
  The first claim follows by applying the fiber homotopy exact sequence to the first row.
  For the second claim
  use the fact that $L \to G_1$ is a fibration (because of $\bigstar$1) with the same fiber as
  $q_n \:  G_0\ang{n} \to G_0$, and apply the fiber homotopy exact sequence to the 
  two rows of the above diagram (together with Five Lemma). 
 \end{proof}
 
 \begin{rem}{\label{R:contractible}}
    In the definition of $\bbG\ang{n}=[L\ang{n-1}\to G_0\ang{n}]$, the fact that $G_0\ang{n}$
    is an $n$-connected cover of $G_0$ is not really needed. All we need is to have a functorial
    replacement $q\: G' \to G$ such that $q$ is a fibration and $\pi_iG'$ is trivial
    for $i\leq n$. For instnace, we could take $G'$ to be the group $\operatorname{Path}_1(G)$
    of paths originating at 1. (In the finite dimensional Lie
    context, however, this would not be a suitable choice as $\operatorname{Path}_1(G)$ is infinite
    dimensional. That is why we chose $G\ang{2}:=\widetilde{G^o}$ instead.)
 \end{rem}
 
\section{Functorial properties of connected covers}{\label{S:Functoriality}}

For $n\leq 2$ we prove that our definition of the $n$-th connected cover 
$\bbG\ang{n}$
of a Lie crossed-modules  is functorial
in Lie butterflies and satisfies the expected adjunction property (Proposition \ref{P:factor}). 
We will need the following property of the connected covers:

 \begin{itemize}
  \item[($\bigstar$2)] For any homomorphism $f \: H \to G$,
   such that $\pi_if \: \pi_iH \to \pi_iG$ is an isomorphism for $0\leq i\leq n-1$, the 
   diagram
         $$\xymatrix@C=8pt@R=6pt@M=6pt{ H\ang{n-1}  
             \ar[rr]^{f\ang{n-1}}\ar[dd]_{q_{n-1}}
                              & & G\ang{n-1}    \ar[dd]^{q_{n-1}}  \\
                                           &     & \\
                       H \ar[rr]_f & & G       }$$     
   is cartesian.
 \end{itemize}

\subsection{Construction of the $n$-th connected cover functor}{\label{SS:construction}} 

Consider the Lie butterfly $B \: \bbH \to \bbG$,
    $$\xymatrix@C=8pt@R=6pt@M=6pt{ H_1 \ar[rd]^{\kappa} \ar[dd]
                          & & G_1 \ar[ld]_{\iota} \ar[dd] \\
                            & E \ar[ld]^{\sigma} \ar[rd]_{\rho}  & \\
                                       H_0 & & G_0       }$$
The butterfly  $B\ang{n} \:  \bbH\ang{n} \to  \bbG\ang{n}$
is defined to be the diagram 
    $$\xymatrix@C=8pt@R=6pt@M=6pt{  L_H\ang{n-1} \ar[rd]^{\kappa_n}  \ar[dd]
                          & &  L_G\ang{n-1} \ar[ld]_{\iota_n}  \ar[dd] \\
                            & F\ang{n-1} \ar[ld]^{\sigma_n} \ar[rd]_{\rho_n} & \\
                                       H_0\ang{n} & & G_0\ang{n}       }$$
Let us explain what the terms appearing in this diagram are. The groups 
$L_G$ and $L_H$ are what we
called $L$ in the definition of the $n$-connected cover (see $\S$ \ref{SS:uniform}}).
For example, $L_H =H_1\times_{H_0}H_0\ang{n}$. The Lie group $F$
appearing in the center of the butterfly is defined to be
      $$F:=H_0\ang{n}\times_{H_0} E \times_{G_0} G_0\ang{n}.$$
The maps $\rho_n$ and $\sigma_n$ are obtained by composing $q_{n-1} \:  F\ang{n-1}
\to F$ with the corresponding projections.   The map $\kappa_n$ is obtained by applying
the functoriality of $(-)\ang{n-1}$ to $(\pr_2,\kappa\circ\pr_1, 1) \: L_H \to F$. 
Definition of $\iota_n$ is 
less trivial and is given in the next paragraphs. We need to show that the kernel of
$\sigma_n \: F\ang{n-1} \to H_0\ang{n}$ is naturally isomorphic to $L_G\ang{n-1}$.

There is an equivalent way of defining $F$ which is somewhat more illuminating. Set 
           $$K:=H_0\ang{n}\times_{H_0} E.$$
Let $\sigma' \: K \to H_0\ang{n}$ be the  first projection map and
$\rho ' \: K \to G_0$  the second projection map composed with $\rho$.
Then,
       $$F= K\times_{\rho', G_0} G_0\ang{n}.$$
Now, observe that we have a short exact
sequence
       $$1 \to  G_1 \llra{\alpha} K \llra{\sigma'} H_0\ang{n} \to 1.$$  
Therefore, we have a cartesian diagram
       $$\xymatrix@C=10pt@R=6pt@M=6pt{  L_G 
             \ar @{^(->} [rr]^{\beta} \ar[dd]_{\pr_1}  
                              & & F \ar[dd]^{\pr_1}   \\
                                           &     & \\
                       G_1 \ar@{^(->}[rr]_{\alpha} & & K       }$$ 
and the sequence
         $$1 \to  L_G \llra{\beta} F \llra{\sigma'\circ\pr_1} H_0\ang{n} \to 1$$                 
is short exact.   (Exactness at the right end follows from  ($\bigstar$1) 
and the fact that $H_0\ang{n}$ is connected.)     
A homotopy fiber sequence argument applied to this short exact sequence shows that
$\alpha$ induces isomorphisms
$\pi_i G_1 \to \pi_i K$, for $0\leq i < n$.  
By  ($\bigstar$2) we   have a cartesian diagram
                $$\xymatrix@C=12pt@R=6pt@M=6pt{  L_G\ang{n-1}
             \ar@{^(->}[rr]^{\beta \ang{n-1}} \ar[dd]_{q_{n-1}}  
                              & & F\ang{n-1} \ar[dd]^{q_{n-1}}   \\
                                           &     & \\
                       L_G \ar@{^(->} [rr]_{\beta} & & F       }$$ 
Therefore, 
          $$1 \to  L_G \ang{n-1} \llra{\beta\ang{n-1}} F\ang{n-1} 
              \llra{\sigma_n} H_0\ang{n} \to 1$$  
is short exact, where $\sigma_n:=\sigma'\circ\pr_1\circ q_{n-1}$. 
(Exactness at the right end follows from  ($\bigstar$1) and the fact that $H_0\ang{n}$
is connected.)     
Setting $\iota_n:=\beta\ang{n-1}$ completes the construction of
our butterfly diagram. The equivariance axioms for this butterfly follow  from the
functoriality  of the $(n-1)$-th connected cover.

 \begin{rem}
    In the case where we have a strict morphism $f \: \bbH \to \bbG$, we can
    define a natural strict morphism 
    $f\ang{n} \: \bbH\ang{n} \to \bbG\ang{n}$
    componentwise. It is natural to ask whether this morphism coincides with the one
    we constructed above using butterflies. The answer is yes. The proof uses the
     the following property of connected covers:
  \begin{itemize}
   \item[($\bigstar$3)]  If $G$ is an $n$-connected group acting on $H$,
      then $(\id,q_{n-1}) \: G\ltimes H\ang{n-1} \to G\ltimes H$ is the $(n-1)$-connected
      cover of $G\ltimes H$. That is, the map $(\id,q_{n-1}) $ is isomorphic to the $q_{n-1}$ 
      map of $G\ltimes H$.
  \end{itemize}
 \end{rem}

\subsection{Effect on the composition of butterflies}{\label{SS:composition}} 

The proof that the construction of the previous subsection respects composition of butterflies
is somewhat intricate.  We will only consider Lie butterflies and assume that $0\leq n\leq 2$, but 
the exact same proofs apply verbatim to topological butterflies
(and also to infinite dimensional Lie butterflies). We begin with a few lemmas. 

 \begin{lem}{\label{L:cartesian}}
  Let $m\geq 0$ be an integer.
  Consider a homotopy cartesian diagram of topological spaces
           $$\xymatrix@C=8pt@R=6pt@M=6pt{ X
             \ar[rr]^h \ar[dd] 
                              & & Y    \ar[dd]^f   \\
                                           &     & \\
                       Z \ar[rr]_g & & W       }$$     
  Suppose that $W$ is $(m+1)$-connected and $Z$ is $m$-connected. Then, $h$ induces
  isomorphisms $\pi_ih \: \pi_iX \to \pi_iY$ for $i \leq m$.
 \end{lem}

 \begin{proof}
  The connectivity assumptions on $Z$ and $T$ imply that the homotopy fiber of $g$ is
  $m$-connected. Since the diagram is homotopy cartesian, the same is true for the
  homotopy fiber of $h$. A homotopy fiber exact sequence implies the claim.
 \end{proof}

 \begin{cor}{\label{C:cartesian}}
  Let $f \: Y \to W$ and $g \: Z \to W$ be homomorphisms of Lie groups and suppose that $W$ is
  $(m+1)$-connected. Suppose that either $f$ or $g$ is a fibration (e.g., surjective).
  Then, we have natural isomorphisms
     $$Z\ang{m} \times_W Y\ang{m}  \cong \big(Z\ang{m} \times_W Y)\ang{m}
               \cong\big(Z\times_W Y\ang{m} )\ang{m}.$$
  In particular, all three groups are $m$-connected.
 \end{cor}
 
 \begin{proof}
   We prove the first equality. Apply Lemma \ref{L:cartesian} to the  diagram
               $$\xymatrix@C=10pt@R=8pt@M=6pt{ X
                     \ar[rr]^h \ar[dd]    & & Y    \ar[dd]^f   \\
                                                   &      & \\
                       Z\ang{m} \ar[rr]_{g\circ q_m} & & W       }$$  
   where $X:=Z\ang{m} \times_W Y$.     The diagram is homotopy cartesian because
   either $f$ or $g\circ q_m$ is a fibration.   Now apply ($\bigstar$2) to $h=\pr_2 \: 
   Z\ang{m} \times_W Y \to Y$.   
 \end{proof}
 
The next lemma is the technical core of this subsection.

 \begin{lem}{\label{L:diamond}}
   Consider the commutative diagram
                $$\xymatrix@C=8pt@R=6pt@M=6pt{ X   \ar[rr]^h \ar[dd]_k 
                                    &      & Y    \ar[dd]^f   \\  
                                     &     & \\
                       Z \ar[rr]_g & & W       }$$    
   of Lie groups. Suppose that $W$ acts on $X$ so that $[f\circ h\:X \to W]$ is a Lie crossed-module.
   Also, suppose that the induced action of $Y$ on $X$  via   $f$ makes the map $h$ 
   $Y$-equivariant      (the action of $Y$ on itself being the
   right conjugation). Assume the same thing
   for the induced action of $Z$ on $X$ via $g$.   Suppose that $W$
   is $(m+1)$-connected, $f$ is surjective, and that $k$ is closed injective normal 
   with $(m+1)$-connected cokernel.  Then, the sequence
     $$\xymatrix@C=12pt@R=8pt@M=6pt{
        1 \ar[r] & X \ang{m} \ar[rr]^(0.35){(k\ang{m},h\ang{m})} 
         && Z\ang{m} \times_W Y\ang{m} 
              \ar[rr]^{u} & & 
       \big(Z\oux{W}{X}Y\big)\ang{m}  \ar[r] & 1 }$$
    is short exact. Here,
    $u$ is the composition $(q_m,\id)\ang{m}\circ \phi$, where $\phi \:  
    Z\ang{m} \times_W Y\ang{m} \risom  
    (Z\ang{m} \times_W Y)\ang{m}$ is the isomorphism of Corollary \ref{C:cartesian}.
    (For the definition of $Z\oux{W}{X}Y$ see the end of $\S$\ref{S:Weak}.). In
    other words, we have a natural isomorphism
         $$Z\ang{m}\oux{W}{X\ang{m}}Y\ang{m} 
            \cong \big(Z\oux{W}{X}Y\big)\ang{m}.$$ 
 \end{lem}

 \begin{proof}
  We start with the short exact sequence
       $$ 1 \to X \to Z\times_W Y \to Z\oux{W}{X} Y \to 1,$$
  which is essentially the definition of  $Z\oux{W}{X} Y$.
  From it we construct the  exact sequence
   $$ 1 \to X\ang{m} \llra{\alpha} 
        Z\ang{m}\times_W Y \llra{(q_m,\id)} Z\oux{W}{X} Y $$
  with $\alpha \:     X\ang{m} \to Z\ang{m}\times_W Y$ is $(k\ang{m},h\circ q_m)$. 
  To see why this sequence is exact, we calculate the kernel of the homomorphism
  $(q_m,\id) \: Z\ang{m}\times_W Y \to Z\times_W Y$: 
      {\small $$X\times_{Z\times_W Y}\big( Z\ang{m}\times_W Y\big)\cong
          X\times_{Z\times_W Y} \big((Z\times_W Y) \times_Z Z\ang{m} \big)\cong X\times_Z
          Z\ang{m}\cong X\ang{m}.$$}
  For the first equality we have used
       $$ Z\ang{m}\times_W Y \cong (Z\times_W Y) \times_Z Z\ang{m}.$$
  For the last equality we have used ($\bigstar$2) for $k \: X \to Z$. (Note that, since
  $\coker k$ is $(m+1)$-connected, $k \: X \to Z$ induces isomorphisms
  on $\pi_i$ for all $i \leq m$.)

  Observe that the last map in the above sequence is a fibration with (open-closed) image $I
   \subseteq Z\oux{W}{X} Y$.
  This fibration has an $m$-connected kernel $X\ang{m}$, so, using the homotopy
  fiber exact sequence, we see that it induces isomorphisms on $\pi_i$ for $i\leq m$. By 
  ($\bigstar$2) we get the cartesian square  
       $$\xymatrix@C=18pt@R=6pt@M=6pt{
          (Z\ang{m} \times_W Y)\ang{m} \ar[dd]_{q_m} 
           \ar[rr]^{(q_m,\id)\ang{m}}
       && \big(Z\oux{W}{X}Y\big)\ang{m} \ar[dd]^{q_m} \\ 
          && \\
        Z\ang{m}\times_W Y \ar[rr] _{(q_m,\id)} &&  I  }$$
  (Note that $I\ang{m}=\big(Z\oux{W}{X}Y\big)\ang{m}$ because
  the $m$-th connected cover only depends on the connected 
  component of the identity, which is contained in $I$.) Precomposing the top
  row with the isomorphism $\phi \:  Z\ang{m} \times_W Y\ang{m} \risom  
  (Z\ang{m} \times_W Y)\ang{m}$
  of Corollary \ref{C:cartesian},  and calling the composition $u$ as in the statement of the
  lemma, we find the following commutative diagram in which  the square on the right is cartesian
     {\small            $$\xymatrix@C=12pt@R=6pt@M=6pt{
        1 \ar[rr]  &&  X \ang{m} \ar[rr]^(0.35){(k\ang{m},h\ang{m})} 
          \ar[dd]^{\id} &&   
        Z\ang{m} \times_W Y\ang{m} \ar[dd]^{(\id,q_m)}     \ar[rr]^{u}
         && \big(Z\oux{W}{X}Y\big)\ang{m} \ar[dd]^{q_m} \ar[rr] &&1 \\ 
         &&  && \\
        1\ar[rr] &&  X \ang{m}   \ar[rr]_(0.4){(k\ang{m},h\circ q_m)} &&  
         Z\ang{m}\times_W Y \ar[rr] _{(q_m,\id)}     &&  I  \ar[rr] && 1 }$$}
  Since the bottom row is short exact, so is the top row. Proof of the lemma is complete.       
 \end{proof}

We need one more technical lemma.

 \begin{lem}{\label{L:elementary}}
   Consider a commutative diagram
                $$\xymatrix@C=8pt@R=6pt@M=6pt{ X
             \ar[rr]^h \ar[dd]_k 
                              & & Y    \ar[dd]^f   \\
                                           &     & \\
                       Z \ar[rr]_g & & W       }$$    
   of topological groups. 
   Suppose that $W$ acts on $X$ so that $[f\circ h\:X \to W]$ is a topological crossed-module.
   Also, suppose that the induced action of $Y$ on $X$  via   $f$ makes the map $h$ 
   $Y$-equivariant (the action of $Y$ on itself being the right conjugation). Assume the same thing
   for the induced action of $Z$ on $X$ via $g$. Suppose that $f$ is surjective. Let 
   $\alpha \: W' \to W$ be a homomorphism with normal image, and denote its cokernel by
   $W_0$. Denote the pullback of the above diagram along $\alpha$ by adding  prime 
   superscripts. Denote the images of $X$ and $Z$ in $W_0$ by $X_0$ and $Z_0$,
   respectively. (Note that $X_0$ is normal in $W_0$.) Then, the sequence
       $$1  \to Z'\oux{W'}{X'}Y' \to Z\oux{W}{X}Y \to Z_0/X_0 \to 1$$
   is exact. In particular, if the image of $W'$ is open in $W$, then $Z'\oux{W'}{X'}Y'$
   is a union of connected components of $Z\oux{W}{X}Y$.
 \end{lem}

 \begin{proof}
   The proof is elementary group theory.
 \end{proof}

We are now ready to prove that our construction of  $n$-connected covers is functorial, that
is, it respects composition of butterflies.

 \begin{prop}{\label{P:compose}}
  Let   $C \: \bbK \to \bbH$ and 
  $B \: \bbH \to \bbG$ be Lie butterflies, and let $B\circ C  \: \bbK \to \bbG$ be their composition. 
  Then, there is a natural isomorphism of butterflies 
  $B\ang{n}\circ C\ang{n}\Rightarrow (B\circ C)\ang{n}$ which makes the assignment 
  $\bbG \mapsto \bbG\ang{n}$  a bifunctor from the bicategory $\LieXM$ of Lie crossed-modules
  and butterflies to itself.
 \end{prop}

 \begin{proof}
  Let $C$  and $B$ be given by
    $$\xymatrix@C=8pt@R=6pt@M=6pt{ K_1 \ar[rd]^{\kappa}  \ar[dd] 
                            & & H_1 \ar[ld]_{\iota}  \ar[dd]  \\
                  & E_C \ar[ld]^{\sigma}   \ar[rd]_{\rho}   & \\
                  K_0 & & H_0       } \ \ \ \ \
        \xymatrix@C=8pt@R=6pt@M=6pt{ H_1 \ar[rd]^{\kappa'}  \ar[dd] 
                            & & G_1 \ar[ld]_{\iota'}  \ar[dd]  \\
                  & E_B \ar[ld]^{\sigma'}  \ar[rd]_{\rho'}  & \\
                  H_0 & & G_0       }$$
   respectively. Then, the composition $B\circ C$ is  the butterfly
      $$\xymatrix@C=10pt@R=4pt@M=4pt{ K_1 \ar[rd]  \ar[dd]
                            & & G_1 \ar[ld]  \ar[dd] \\
                  & E_C\oux{H_0}{H_1}E_B \ar[ld]  \ar[rd]   & \\
                  K_0 & & G_0       }$$

  Recall the notation of $\S$\ref{SS:construction}:   
     $$F_C=K_0\ang{n}\times_{K_0}E_C\times_{H_0}H_0\ang{n} \ \ \text{and} \ \
         F_B=H_0\ang{n}\times_{H_0}E_B\times_{G_0}G_0\ang{n}.$$
  The group appearing in the center of the butterfly $B\ang{n}\circ C\ang{n}$ is
     $$F_C\ang{n-1}\oux{H_0\ang{n}}{L_H\ang{n-1}}F_B\ang{n-1}.$$
  The group appearing in the center of the butterfly   $(B\circ C)\ang{n}$ is
    $$F_{B\circ C}\ang{n-1}=
      \big(K_0\ang{n}\times_{K_0}(E_C\oux{H_0}{H_1}E_B)\times_{G_0}G_0\ang{n}\big)\ang{n-1}.$$
  We show that there is a natural isomorphism from the former to the latter.
  For this, first we apply Lemma \ref{L:diamond} with $m=n-1$ and
      $$X=L_H, \ \ Z=F_C, \ \ Y=F_B, \ \ \text{and} \ \  W=H_0\ang{n}$$
  to get
     $$F_C\ang{n-1}\oux{H_0\ang{n}}{L_H\ang{n-1}}F_B\ang{n-1}
       \cong  (F_C\oux{H_0\ang{n}}{L_H}F_B)\ang{n-1}.$$   
  It is now enough to construct a natural isomorphism
     $$F_C\oux{H_0\ang{n}}{L_H}F_B\to  F_{B\circ C},$$
  that is,
    $$(K_0\ang{n}\times_{K_0}E_C\times_{H_0}H_0\ang{n} )
        \oux{H_0\ang{n}}{L_H}(H_0\ang{n}\times_{H_0}E_B\times_{G_0}G_0\ang{n})$$ 
    $$\longrightarrow \big(K_0\ang{n}\times_{K_0}(E_C\oux{H_0}
       {H_1}E_B)\times_{G_0}G_0\ang{n}\big).$$
  This, however, may not be the case. 
  More precisely, there is such a natural homomorphism, but it is not necessarily
  an isomorphism. It is, however, an isomorphism between the connected components of the
  identity elements (and that is enough for our purposes). To see this, use  
  Lemma \ref{L:elementary} with
      $$X=H_1, \ \, Z=K_0\ang{n}\times_{K_0} E_C,\ \ Y=E_B\times_{G_0}G_0\ang{n},$$
      $$W=H_0, \ \ W'=H_0\ang{n}, \ \ 
           \text{and} \ \  \alpha=q_n.$$
  (Recall that $L_H=H_1\times_{H_0}H_0\ang{n}$.) Here we are using the fact that
  $\alpha=q_n \: H_0\ang{n} \to H_0$ surjects onto 
  the connected component of the identity element in $H_0$.
 
  We omit the verification that the isomorphism 
  $B\ang{n}\circ C\ang{n}\Rightarrow (B\circ C)\ang{n}$ respects 
  isomorphisms of butterflies and that it commutes with the associator isomorphisms in
  $\LieXM$.
 \end{proof}

The proposition is valid in the topological setting as well, and the proof is identical.

 \begin{cor}{\label{C:cover}} Let $f\: \bbH \to \bbG$ be an equivalence of Lie crossed-modules.
  Then the induced morphism $f\ang{n} \: \bbH\ang{n} \to \bbG\ang{n}$, $n=0,1,2$, is also an
  equivalence of Lie crossed-modules. 
 \end{cor}

\subsection{Adjunction property of connected covers}{\label{SS:adjunction}} We show that 
$n$-connected covers of Lie crossed-modules satisfy the expected adjunction property, namely, 
that a weak morphism $f \: \bbH \to \bbG$ from an $n$-connected Lie crossed-module $\bbH$
uniquely factors through $q_n \: \bbG\ang{n} \to \bbG$ (Proposition \ref{P:factor}). 

As in the previous section, we will assume that $n\leq 2$. What we say remains 
valid for  topological crossed-modules (and also for infinite
dimensional Lie crossed-modules). We will use  
the adjunction property for groups: 
 \begin{itemize}
  \item[($\bigstar4$)]  For any homomorphism $f \: H \to G$  with $H$
  $(n-1)$-connected, $f$ factors uniquely through $q_{n-1} \: H\ang{n-1} \to H$.
 \end{itemize}

 \begin{prop}{\label{P:factor}}
   Let $\bbG$ and $\bbH$ be Lie crossed-modules, and suppose
   that $\bbH$ is $n$-connected (Definition \ref{D:connected}).
   Then, the morphism 
   $q=q_n \:  \bbG\ang{n} \to \bbG$  induces an equivalence of hom-groupoids
                 $$q_*\: \LieXM(\bbH,\bbG\ang{n})\risom \LieXM(\bbH,\bbG).$$  
 \end{prop}
 
 \begin{proof}
  We construct an inverse  functor (quasi-inverse, to be precise)  to $q_*$. The construction 
  is very similar to the construction of the $n$-connected cover of a butterfly given in the previous
  subsection.
 
  Since $\bbH=[H_1 \to H_0]$ is $n$-connected, we may assume that $H_0$ is $n$-connected
  and $H_1$ is $(n-1)$-connected (this was discussed in $\S$ \ref{SS:definition}).  
  Consider  a butterfly $B$
         $$\xymatrix@C=8pt@R=6pt@M=6pt{ H_1 \ar[rd]^{\kappa} \ar[dd]
                          & & G_1 \ar[ld]_{\iota} \ar[dd] \\
                            & E \ar[ld]^{\sigma} \ar[rd]_{\rho}  & \\
                                       H_0 & & G_0       }$$
  in $\LieXM(\bbH,\bbG)$. Define $F:=E\times_{G_0}G_0\ang{n}$. Let 
  $\tau \: F \to H_0$ be $\sigma\circ\pr_1$. Since $\tau$ is a (locally trivial)
  fibration and $H_0$ is connected, $\tau$ is surjective. 
  On the other hand,  $\ker\tau$ is the inverse image of $\iota(G_1)$
  under the projection $\pr_1 \: F \to E$; this is exactly 
  $G_1 \times_{G_0} G_0\ang{n}=L_G$. That is, we have a short
  exact sequence
     $$1 \to  L_G \llra{\beta} F \llra{\tau} H_0 \to 1.$$ 
  It follows from  ($\bigstar$2) applied to $\beta$ that the sequence
     $$1 \to  L_G\ang{n-1} \llra{\beta\ang{n-1}} F\ang{n-1} 
  \llra{\tau\circ q_{n-1}} H_0 \to 1$$ 
  is also short exact. 
 
 Define  the butterfly $B'$ to be
             $$\xymatrix@C=10pt@R=8pt@M=6pt{ H_1 \ar[rd]^{\kappa'} \ar[dd]
                          & & L_G\ang{n-1} \ar[ld]_{\beta\ang{n-1}} \ar[dd] \\
                            & F\ang{n-1} \ar[ld]^{\tau\circ q_{n-1}}
                                  \ar[rd]_{\varrho}  & \\
                                       H_0 & & G_0\ang{n}       }$$
  where $\varrho=\pr_2\circ q_{n-1}$ and $\kappa'$ is obtained by the adjunction  
  property ($\bigstar$4) applied to $q_{n-1} \: F\ang{n-1} \to F$.

  It is easy to verify that $B \mapsto B'$ is an inverse to $q_*$. (For this, use the fact
  that $E$ is the pushout of $F\ang{n-1}$ along 
  $\pr_1\circ q_{n-1}\:L_G\ang{n-1}\to G_1$ and apply \cite{Maps}, $\S$10.2.)       
 \end{proof}
 
 \begin{cor}{\label{C:connectedadjunction}}  
  For $n=0,1,2$, the inclusion of the full sub bicategory of $\LieXM$ consisting of $n$-connected
  Lie crossed-modules is left adjoint to the $n$-connected cover bifunctor
  $(-)\ang{n} \: \LieXM \to \LieXM$.
 \end{cor}
 
\section{The bifunctor from Lie crossed-modules to \lie2s}{\label{S:Bifun}}

In this section we prove our main integration results for weak morphisms of \lie2s (Theorem
\ref{T:bifunctor}} and Corollary \ref{C:adjoint}).
Throughout the section, we fix the base ring to be  $\mathbb{R}$ or $\mathbb{C}$. All Lie groups
are finite dimensional (real or complex, respectively). 
  
We begin with a simple lemma.
 
 \begin{lem}{\label{L:equivariant}}
    Let $H$, $K$ and $K'$ be connected Lie groups.  Suppose that $H$ acts on $K$ and $K'$ by
    automorphisms and let $f \: K \to K'$ be a Lie homomorphism. If the induced map
    $\Lie f \: \Lie K \to \Lie K'$ is $H$-equivariant then so is $f$ itself.  
 \end{lem}

 \begin{proof}
    This follows from the fact that if two group homomorphisms induce the same 
    map on Lie algebras then they are equal. 
 \end{proof}
 
Let $\LieAlgXM$  be the full sub bicategory of  $\TwoTermB$ consisting of strict \lie2s (i.e., Lie
algebra crossed-modules).
 
 \begin{thm}{\label{T:bifunctor}} 
   Taking Lie algebras induces a bifunctor 
         $$\Lie \: \LieXM \to \TwoTermB.$$ 
   The bifunctor $\Lie$ factors through and essentially surjects onto 
   $\LieAlgXM$. Furthermore, for $\bbH, \bbG \in \LieXM$, the induced functor
         $$\Lie \: \LieXM(\bbH,\bbG)\to \TwoTermB(\Lie\bbH,\Lie\bbG)$$
     on   hom-groupoids is  
   \begin{itemize}
      \item[$i$)] faithful, if $\bbH$ is connected;
      \item[$ii$)] fully faithful, if $\bbH$ is 1-connected;
      \item[$iii$)] an equivalence, if $\bbH$ is 2-connected.
   \end{itemize}
 \end{thm}

 \begin{proof} That $\Lie\: \LieXM \to \TwoTermB$ 
  is a bifunctor follows from the fact that taking Lie algebras is exact and
  commutes with fiber products of Lie groups.

 \medskip  
  
 \noindent{\em Proof of} ($i$). Let $\bbG=[G_1 \to G_0]$ and $\bbH=[H_1\to H_0]$.
  Let $B,B' \: \bbH \to \bbG$ be two butterflies. Since $\bbH$ is connected,
  we may assume that $H_0$ is connected (see $\S$\ref{SS:definition}). Denote the
  NE-SW short exact sequences for $B$ and $B'$ by
    $$0\to G_1 \to E \to H_0 \to 0,$$
    $$0\to G_1 \to E' \to H_0 \to 0.$$
  Consider two isomorphisms $B \Rightarrow B'$, given by $\Phi,\Psi \: E \to  E'$,  such  that 
     $$\Lie \Phi=\Lie \Psi \: \Lie E \to \Lie E'.$$ 
  Then, $\Phi$ and $\Psi$ 
  are equal on the connected component $E^o$ and also on $G_1$. Since $H_0$ 
  is connected, $E^o$ and $G_1$ generate $E$, so $\Phi$ and $\Psi$ 
  are equal on the whole $E$.
 
 \medskip  
  
 \noindent{\em Proof of} ($ii$).  
  Notation being as in the previous part, we may assume that $H_0$ is connected and
  simply-connected and
  $H_1$ is connected  (see $\S$\ref{SS:definition}). 
  Consider  an isomorphism $\Lie B \Rightarrow \Lie B'$ given  by $f\: \Lie E \to \Lie E'$.
  We show that $f$ integrates to $\Phi \: E \to E'$.
  
  Let $\tilde{E} \to E$ be the universal cover of $E$. 
  Integrate $f$ to a homomorphism  $\tilde{\Phi} \: \tilde{E} \to E'$. Consider the diagram
     $$\xymatrix@C=36pt@R=18pt@M=4pt{ 
         &&\tilde{E} \ar[d]^{\beta} \ar@/_1.5pc/ [ddd]_(0.3){\tilde{\Phi}} 
              \ar[rdd]^{\gamma}  && \\
         & G \ar[d]_{\alpha} \ar[ru]^{\delta} & E \ar@{..>} [dd]^{\Phi} \ar[dr]^(0.3){\sigma} &  & \\
              0 \ar[r] & G_1 \ar[ur]  |!{[uur];[dr]}\hole \ar[rd] & &  H_0 \ar[r]& 0 \\
         & & E' \ar[ur] && 
        }$$
  Here $G$ is the kernel of $\gamma:=\sigma\beta \: \tilde{E} \to  H_0$. Note that
  $G\cong G_1\times_{E}\tilde{E}$. That is, $G$ is the pullback of $\tilde{E}$ along
  the map $G_1 \to E$. Since  
  $\pi_iH_0=0$ for $i=1,2$, a fiber homotopy exact sequence argument shows  
  that $\pi_1G_1 \to \pi_1E$ is an isomorphism. Hence, $G$ is the universal cover
  of $G_1$ and, in particular, is connected.
   
  If we apply $\Lie$ to the above diagram, we obtain a commutative diagram of Lie algebras. 
  Therefore, since all the groups involved are connected, 
  the original diagram of Lie groups is also commutative. Since the top left
  square is cartesian, $\delta$ induces an isomorphism $\delta \: \ker\alpha \to \ker\beta$.
  Commutativity of the diagram implies then that $\tilde{\Phi}$ vanishes on $\ker\beta$. Therefore,
  $\tilde{\Phi}$ induces a homomorphism $\Phi \: E\to E'$ which makes the diagram commute.
    
  By looking at the corresponding Lie algebra maps,
  we see that if $f$ commutes with the other two maps of the butterflies, then so does 
  $\Phi$. That is, $\Phi$ is indeed a morphism of butterflies from $B$ to $B'$.
   
 \medskip  
  
 \noindent{\em Proof of} ($iii$). 
  We may assume that $H_0$ and $H_1$ are  connected and
  simply-connected  (see $\S$\ref{SS:definition}).
  In view of the previous part, we have to show that every butterfly $B \: \Lie\bbH \to \Lie\bbG$,
              $$\xymatrix@C=8pt@R=6pt@M=6pt{ \Lie H_1 \ar[rd]^{\kappa} \ar[dd]
                          & & \Lie G_1 \ar[ld]_{\iota} \ar[dd] \\
                            & E \ar[ld]^{\sigma} \ar[rd]_{\rho}  & \\
                                       \Lie H_0 & & \Lie G_0      }$$
  integrates to a butterfly $\Int B \: \bbH \to \bbG$. Let $\Int E$ be the simply-connected
  Lie group whose Lie algebra is $E$. Let $G$ be the kernel of $\Int\sigma \: \Int E \to
   H_0$. Since $\pi_i H_0=0$, $i=0,1,2$, an easy homotopy fiber exact sequence argument 
  implies that $G$ is connected and simply-connected. 
   
  We identify the Lie algebras of $G$ and $G_1$ via $\iota \: \Lie  G_1 \to E$ and regard them as 
  equal.  Since $G$ is simply-connected and $G_1$ is connected, we have a natural isomorphism
  $\bar{\iota} \: G_1 \to G/N$ for some discrete central subgroup $N \subseteq G$. 
  We claim that $N$ is a normal subgroup of $\Int E$. To prove this, we compare the 
  conjugation action of $\Int E$ on $G$ with the action of $\Int E$ on $G_1$ obtained via
  $\Int\rho \: \Int E \to  G_0$. (The latter is the integration of the Lie algebra homomorphism
  $\rho \: E \to \Lie G_0$.) The equivariance axiom of the butterfly for the map $\rho$,
  plus the fact that $\bar{\iota}^{-1}\circ\pr \: G \to G_1$ induces the 
  identity map on the Lie algebras, implies  (Lemma \ref{L:equivariant}) that 
  $\bar{\iota}^{-1}\circ\pr \: G \to G_1$ is $\Int E$-equivariant.
  Therefore, its kernel $N$ is invariant under the conjugation  action of $\Int E$. That is,
  $N \subseteq \Int E$ is normal.

  An argument similar to the one used in the previous part shows that the map 
  $\Int\rho \: \Int E \to G_0$ vanishes on  $N$. More precisely, repeat the same 
  argument with the diagram
    $$\xymatrix@C=36pt@R=18pt@M=4pt{ 
     &&\Int E \ar[d]\ar@/_2.5pc/ [ddd]_(0.3){\Int\rho} 
              \ar[rdd]^{\Int\sigma}  &&\\
      & G \ar[d]_{\alpha} \ar[ru]^{\delta} & (\Int E)/N \ar@{..>} [dd]^{\bar{\rho}}
        \ar[dr]_{\bar{\sigma}}  &  & \\
     0 \ar[r] & G_1 \ar[ur]_(0.3){\bar{\iota}}  |!{[uur];[dr]}\hole \ar[rd]_{\partial} & &  H_0 \ar[r]& 0 \\
     & & G_0  && 
     }$$
  
  Thus, we obtain an induced homomorphism $\bar{\rho} \: (\Int E)/N \to G_0$. Denote
  the map $(\Int E)/N \to H_0$ induced from $\Int\sigma$ by $\bar\sigma$. Collecting what we
  have so far, we obtain a partial butterfly diagram
                  $$\xymatrix@C=8pt@R=6pt@M=6pt{ 
                          & &  G_1 \ar[ld]_{\bar{\iota}} \ar[dd] \\
                            & (\Int E)/N \ar[ld]^{\bar{\sigma}} \ar[rd]_{\bar{\rho}}  & \\
                                       H_0 & &  G_0      }$$
   (Observe that applying $\Lie$ to this partial butterfly gives us back the corresponding 
   portion of the original butterfly $B$.) Finally, using  the fact that $H_1$ is connected and 
   simply-connected, we can complete the butterfly by integrating $\kappa$ to 
   $\bar{\kappa} \: H_1 \to  (\Int E)/N $. It is easily verified that the resulting diagram 
   satisfies the butterfly axioms -- this is the sought after butterfly $\Int B \: \bbH \to \bbG$. 
   The proof is complete.
 \end{proof}

Let $V \mapsto \Int V$ denote the functor which associates to a Lie algebra $V$ the
corresponding connected simply-connected Lie group. For 
a crossed-module in Lie algebras $\bbV=[V_1 \to V_0]$ we denote the corresponding 
crossed-module in Lie groups  by $\Int \bbV:=[\Int V_1 \to \Int V_0]$.

 \begin{cor}{\label{C:adjoint}}  
    The bifunctor $\Int \: \LieAlgXM \to \LieXM$ is left adjoint to the bifunctor
     $\Lie \: \LieXM \to \LieAlgXM$.
 \end{cor}

 \begin{proof}
   This follows immediately from Theorem \ref{T:bifunctor}.
 \end{proof}

 \begin{rem}{\label{L:extendadjunction}}
  Presumably, the adjunction of Corollary \ref{C:adjoint} can be
  extended to
                  $$\xymatrix@C=-24pt@R=-10pt@M=14pt{
            \Int \:   \LieAlgXM \ar@{^(->} [dd]
                              & \leftrightharpoons & \LieXM  : \Lie \ar@{^(->}  [dd]   \\
                                           &     & \\
                        \Int \: \TwoTerm  & \leftrightharpoons & \mathbf{Lie2Gp} : \Lie      }$$   
  Here, by $\mathbf{Lie2Gp}$ we mean the 2-category of Lie group stacks. The inclusion on the right
  is given by the fully faithful bifunctor 
         $$\LieXM \to \mathbf{Lie2Gp},$$
         $$[G_1 \to G_0] \mapsto [G_0/G_1].$$
 \end{rem}
 
\section{Appendix: functorial $n$-connected covers for $n\geq 3$}{\label{A:1}}

Axioms ($\bigstar$1-4)  discussed in $\S$\ref{S:Connected}-
\ref{S:Functoriality} have a certain iterative property which we would like to point out in
this appendix. To simplify the notation, we will replace $n-1$ by $m$. 

We saw in $\S$\ref{S:Connected}-
\ref{S:Functoriality} that, for  $m\leq 1$, the standard choices for the 
$m$-connected cover functors $(-)\ang{m}$ on the category of topological groups automatically
satisfy ($\bigstar$1-4). Using this we constructed our $m$-connected cover
bifunctor $(-)\ang{m}$ on the bicategory of topological (or Lie)  crossed-modules
for $m\leq 2$. It can be
shown that these bifunctors again satisfy (a categorified version) of
($\bigstar$1-4).
 
A magic seems to have occurred here: we managed to raise
$m$ from 1 to 2! This may sound contradictory, as we do not to expect to have a functorial
$2$-connected cover functor $(-)\ang{2}$ on the category of topological groups which satisfies
either the pullback property ($\bigstar$2) or the adjunction property ($\bigstar$4).  
 
This apparent contradiction is explained by noticing  that our definition of
$(-)\ang{2}$ indeed yields a crossed-module, even if the input is a topological group. More
precisely, for  a topological group $G$, we get
        $$G\ang{2}=[\widetilde{L^o} \to G'],$$
where $q \: G' \to G$ is a choice of a 2-connected replacement for $G$ and $L=\ker q$.
(For example, take $G'=\operatorname{Path}_1(G)$, the space of paths starting at $1$;
see Remark \ref{R:contractible}.) It is also 
interesting to note that, for different choices of 2-connected 
replacement $q \: G' \to G$, the resulting
crossed-modules $G\ang{2}$ are canonically (up to a unique isomorphism of butterflies) 
equivalent.

The upshot of this discussion is that, 2-connected covers of topological groups seem to more 
naturally exist as topological crossed-modules. Another implication is that  we can now iterate
the process. For example, we get a functorial
construction of a 3-connected cover $G\ang{3}$ of a topological group $G$ as a
2-crossed-module, and this (essentially unique) construction enjoys a categorified
version of ($\bigstar$1-4).

This seems to hint at the following general philosophy: for any $m\leq k+1$, 
there should be a (essentially unqiue) definition of $m$-connected covers $\bbG\ang{m}$ for 
topological $k$-crossed-modules $\bbG$ which enjoys a categorified version of ($\bigstar$1-4).

We point out that the notion of  $k$-crossed-module exists for $k\leq 3$
(see \cite{Conduche} and \cite{ArKuUs}). The butterfly construction of the tricategory of 
$2$-crossed-modules (and week morphisms) is being developed in \cite{ButterflyIII}. For higher
values of $k$ the simplicial approach is perhaps a better alternative, as $k$-crossed-modules
tend to become immensly complicated as $k$ increases.

\begin{rem}
  The above discussion applies to the case where we 
  replace topological groups with infinite dimensional
  Lie groups. 
\end{rem}

\providecommand{\bysame}{\leavevmode\hbox
to3em{\hrulefill}\thinspace}
\providecommand{\MR}{\relax\ifhmode\unskip\space\fi MR }
\providecommand{\MRhref}[2]{%
  \href{http://www.ams.org/mathscinet-getitem?mr=#1}{#2}
} \providecommand{\href}[2]{#2}

\end{document}